\documentclass[reqno]{amsart}
\usepackage{amsmath,amsthm,amscd,amssymb,amsfonts, amsbsy}
\usepackage{latexsym, color, enumerate}
\usepackage{pxfonts}

\usepackage[dvipsnames]{xcolor}
\usepackage{soul}
\numberwithin{equation}{section}

\newtheorem{theorem}{Theorem}[section]

\newtheorem{lemma}[theorem]{Lemma}

\theoremstyle{definition}

\theoremstyle{definition}

\theoremstyle{definition}

\newcommand{\supp}{\operatorname{supp}}
\newcommand{\dist}{\operatorname{dist}}
\newcommand{\diam}{\operatorname{diam}}

\newcommand{\bR}{\mathbb{R}}

\providecommand{\set}[1]{\{#1\}}
\providecommand{\Set}[1]{\left\{#1\right\}}

\providecommand{\abs}[1]{\lvert#1\rvert}
\providecommand{\Abs}[1]{\left\lvert#1\right\rvert}

\providecommand{\norm}[1]{\lVert#1\rVert}

\renewcommand{\vec}[1]{\boldsymbol{#1}}
\renewcommand{\qedsymbol}{$\blacksquare$}

\begin{document}
\title{Green's function for nondivergence elliptic operators in two dimensions}

\author[H. Dong]{Hongjie Dong}
\address[H. Dong]{Division of Applied Mathematics, Brown University,
182 George Street, Providence, RI 02912, United States of America}
\email{Hongjie\_Dong@brown.edu}
\thanks{H. Dong was partially supported by the NSF under agreement DMS-1600593.}

\author[S. Kim]{Seick Kim}
\address[S. Kim]{Department of Mathematics, Yonsei University, 50 Yonsei-ro, Seodaemun-gu, Seoul 03722, Republic of Korea}
\email{kimseick@yonsei.ac.kr}
\thanks{S. Kim is partially supported by National Research Foundation of Korea(NRF) under agreement   NRF-20151009350 and NRF-2019R1A2C2002724.}



\begin{abstract}
We construct the Green function for second-order elliptic equations in non-divergence form when the mean oscillations of the coefficients satisfy the Dini condition.
We show that the Green's function is BMO in the domain and establish logarithmic pointwise bounds.
We also obtain pointwise bounds for first and second derivatives of the Green's function.
\end{abstract}
%\today
\maketitle

\section{Introduction and main results}

We consider a second-order elliptic operator in a bounded, connected, and open set $\Omega$ in $\bR^2$.
Let $L$ be an elliptic operator in non-divergence form given by
\begin{equation}				\label{master}
Lu=\sum_{i,j=1}^2 a^{ij}(x) D_{ij}u.
\end{equation}
Here, we assume (without loss of generality) that the coefficients $a^{ij}$ are symmetric and defined on the entire space $\bR^2$.
We require that the matrix $\mathbf{A}=(a^{ij})$ satisfy the uniform ellipticity condition, i.e., there exists a constant $\nu \in (0,1]$ such that
\begin{equation}				\label{ellipticity}
\nu \abs{\xi}^2 \le \sum_{i,j=1}^2 a^{ij}\xi_i \xi_j \le \nu^{-1} \abs{\xi}^2.
\end{equation}
We shall say that $\mathbf{A}=(a^{ij})$ are of Dini mean oscillation
in $\Omega$ if the mean oscillation function $\omega_{\mathbf A}: \bR_+ \to \bR$ defined by
\begin{equation}					\label{13.24f}
\omega_{\mathbf A}(r):=\sup_{x\in \overline \Omega} \fint_{\Omega(x,r)} \,\abs{
\mathbf A(y)-\bar {\mathbf A}_{\Omega(x,r)}}\,dy,
\end{equation}
where
\[
\bar{\mathbf A}_{\Omega(x,r)} :=\fint_{\Omega(x,r)} \mathbf A = \frac{1}{\abs{\Omega(x,r)}}\int_{\Omega(x,r)} \mathbf A
\quad\text{and}\quad \Omega(x,r)=B(x,r) \cap \Omega,
\]
satisfies the Dini condition, i.e.,
\begin{equation}			\label{dini}
\int_0^1 \frac{\omega_{\mathbf A}(t)}t \,dt <+\infty.
\end{equation}

In this article, we shall show that if $\Omega$ is a bounded domain with regular boundary and the coefficients $\mathbf{A}=(a^{ij})$ are of Dini mean oscillation in $\Omega$, then the Green's function $G(x,y)$ exists and has logarithmic pointwise bounds.
In fact, we show that $G(\cdot, y)$ is BMO in $\Omega$ and $G(\cdot,\cdot)$ is continuous on $(\Omega\times \Omega) \setminus \{(x,x):x\in \Omega\}$.
We shall also derive pointwise bounds for $G(x,y)$ as well as its derivatives $D_x G(x, y)$ and $D^2_x G(x,y)$.

It is well known that the elliptic operators in divergence form admit Green's functions that are comparable to those of the Laplace operator, even when the coefficients are just measurable; see \cite{LSW63, GW82, KN85, CL92}.
There are also many papers in the literature dealing with the existence and estimates of Green's functions or fundamental solutions of non-divergence form elliptic operators with measurable or continuous coefficients; see e.g., \cite{Bauman84b, Bauman85, Esc2000, FS84, Krylov92}.
In the case when the coefficient matrix $\mathbf{A}$ is uniformly H\"older (or Dini) continuous, it is well known that a Green's function is continuous and satisfies the pointwise bound comparable to that of Laplace operator; see e.g., \cite{Miranda, S73, An78}. For parabolic operators, we refer the reader to \cite{Friedman} for the construction of fundamental solutions by the parametrix method, and also \cite{Eidelman, C11}.
However, unlike the Green's function for elliptic operators in divergence form, in general, Green's function for non-divergence form elliptic operators do not necessarily have pointwise bounds, even if the domain is smooth and the coefficients are uniformly continuous; see \cite{Bauman84}.

In a recent paper \cite{HK20}, it is shown that in three dimensions or higher, the Green's functions for non-divergence form elliptic operators have the pointwise bound $c\abs{x-y}^{2-n}$ if the coefficients of the operator have Dini mean oscillations.
See also \cite{MM10} for a related result with coefficients satisfying ``square'' Dini condition.
However, the proof in \cite{HK20} does not work in the two dimension, which is mostly due to the failure of Sobolev embedding $W^{1,2} \hookrightarrow L^{2n/(n-2)}$ when $n=2$.
It is worthwhile to mention that even for the Laplace case, the behavior of Green's functions is  different in two dimensions.
As a matter of fact, there are quite a few papers devoted to the study of two dimensional Green's functions; see e.g. \cite{DK09, OKB15, TKB13}.

The adjoint operator $L^\ast$ is given by
\[
L^* u= \sum_{i,j=1}^2 D_{ij} (a^{ij}(x) u),
\]
where the coefficient matrix $\mathbf{A}=(a^{ij})$ is the same as that of $L$ and thus, it is of Dini mean oscillation in $\Omega$.
It is known that if $f \in L^p(\Omega)$ with $p \in (1,\infty)$, then the unique $L^p$ solution (see, e.g., \cite[Lemma~2]{EM2016}) of the problem
\begin{equation}				\label{eq20.40th}
L^\ast u = f\;\text{ in }\;\Omega,\quad u=0 \;\text{ on }\;\Omega,
\end{equation}
is uniformly continuous in $\Omega$; see Theorem~1.8 in \cite{DEK18}.
The definitions of $\mathsf{BMO}(\Omega)$, $\norm{\cdot}_\ast$,  $H^1$ atom in $\Omega$, etc. are given in the next section.

Now, we state our main results.
In our first theorem, we shall assume, in place of \eqref{13.24f}, that the $L^2$ mean oscillation
\begin{equation}					\label{13.24g}
\omega_{\mathbf A}(r):=\sup_{x\in \overline \Omega} \left( \fint_{\Omega(x,r)} \,\abs{
\mathbf A(y)-\bar {\mathbf A}_{\Omega(x,r)}}^2\,dy\right)^{\frac12}
\end{equation}
satisfies the Dini condition \eqref{dini}.
In light of H\"older's inequality, this assumption is stronger than our hypothesis that $\mathbf A$ is of Dini mean oscillation.
We will return to the original definition \eqref{13.24f} in our second theorem.

\begin{theorem}				\label{thm1}
Let $\Omega$ be a bounded $C^{1,1}$ domain in $\bR^2$.
Assume the coefficient $\mathbf{A}=(a^{ij})$ of the operator $L$ in \eqref{master} satisfies the uniform ellipticity condition \eqref{ellipticity} and that the $L^2$ mean oscillation \eqref{13.24g} of $\mathbf A$ satisfies the Dini condition \eqref{dini}.
Then, there exists a Green's function $G(x,y)\ ( \text{for any } \ x, y \in \Omega, \ x\neq y)$ and it is unique in the following sense:
if $u$ is the unique adjoint solution of the problem \eqref{eq20.40th}, where $f \in L^p(\Omega)$ with $p>1$, then $u$ is represented by
\begin{equation}				\label{eq1747m}
u(y)=\int_\Omega G(x,y) f(x)\,dx.
\end{equation}
Also, $G^\ast(x,y)=G(y,x)$ becomes the Green's function for the adjoint operator $L^\ast$, which is characterized as follows:
for $q>1$ and $f\in L^{q}(\Omega)$,
if $v \in W^{2,q}(\Omega)\cap W^{1,q}_0(\Omega)$ is the strong solution of
\begin{equation}
                                \label{eq10.36}
Lv=f \;\text{ in }\;\Omega,\quad v=0\;\text{ on }\;\partial \Omega,
\end{equation}
then, we have the representation formula
\begin{equation}				\label{eq1748m}
v(y)=\int_\Omega G^\ast(x,y)f(x)\,dx=\int_\Omega G(y,x)f(x)\,dx.
\end{equation}
The Green function $G(x,y)$ satisfies the following  estimates:
\begin{gather}
					\label{eq0528af}
G(\cdot, y) \in \mathsf{BMO}(\Omega)\quad\text{with}\quad \norm{G(\cdot, y)}_{\ast} \le C,\quad  y \in \Omega,\\
					\label{eq0529af}
\abs{G(x,y)} \le C\left( 1+ \log \frac{\diam \Omega}{\abs{x-y}} \right),\quad x \neq y \in \Omega,\\
					\label{eq0530af}
\abs{D_x G(x,y)} \le \frac{C}{\abs{x-y}}, \quad x \neq y \in \Omega,\\
					\label{eq408}
\abs{D_x^2 G(x,y)} \le \frac{C}{\abs{x-y}^2}, \quad y \in \Omega,\ x\in B(y,d_y/2)\setminus \set{y},
\end{gather}
where $d_y=\dist (y,\partial\Omega)$ and $C$ depends on $\nu$, $\Omega$, and the $L^2$ mean oscillation $\omega_{\mathbf A}$ as in \eqref{13.24g}.
\end{theorem}

In our second theorem, we drop the extra assumption that $\mathbf A$ is of $L^2$ Dini mean oscillation but we shall instead assume that $\Omega$ is a $C^{2,\alpha}$ domain for some $\alpha>0$.
\begin{theorem}				\label{thm3}
Let $\Omega$ be a bounded $C^{2,\alpha}$ domain in $\bR^2$.
Assume the coefficient $\mathbf{A}=(a^{ij})$ of the operator $L$ in \eqref{master} satisfies the uniform ellipticity condition \eqref{ellipticity} and is of Dini mean oscillation in $\Omega$.
Then, all conclusions of Theorem~\ref{thm1} are valid.
Moreover, we have
\begin{equation}		\label{eq0531af}
\abs{D_x^2 G(x,y)} \le C \abs{x-y}^{-2},\quad x \neq y \in \Omega,
\end{equation}
where  $C$ depends on $\nu$, $\Omega$, and the mean oscillation $\omega_{\mathbf A}$ as in \eqref{13.24f}.
\end{theorem}

Our last theorem is the key to the proof of Theorems \ref{thm1} and \ref{thm3}, the statement of which has its own interest.
\begin{theorem}			\label{thm2}
Let $\Omega \subset \bR^2$ and the coefficient $\mathbf A$ satisfy the conditions of  Theorem~\ref{thm1} (resp. Theorem~\ref{thm3}) with $\omega_{\mathbf A}$ given by the formula \eqref{13.24g} (resp. \eqref{13.24f}).
Then, there exists a constant $N_0=N_0(\nu, \Omega, \omega_{\mathbf A})$ such that if $u$ is the adjoint solution of the problem
\[
L^* u=a\;\text{ in }\;\Omega,\quad u=0\;\text{ on }\;\partial\Omega,
\]
where $a$ is an $H^1$ atom in $\Omega$, then $\abs{u} \le N_0$.
\end{theorem}

The proof of the theorem heavily relies on the assumption that $\Omega \subset \bR^2$ and cannot be applied to higher dimension.
As a matter of fact, Theorem~{\ref{thm2}} is not true in higher dimensions since otherwise it would imply that the Green's function $G(\cdot, y)$ belongs to $\mathsf{BMO}(\Omega)$, which is not true in higher dimensions.
For example,  consider the Green's function $G(\cdot, y)$ for Laplace operator in $\Omega=B(0,1) \subset \bR^n$ with $n \ge 3$.
It does not belong to $L^p(\Omega)$ for $p\ge \frac{n}{n-2}$ and hence cannot belong to $\mathsf{BMO}(\Omega)$.

We conclude the introduction with a few remarks.
First of all, the Green's function $G(x,y)$ is continuous on $\set{(x,y) \in \overline\Omega \times \overline \Omega: x\neq y}$.
Indeed, the proof of Theorem~\ref{thm1} will show that $L^\ast G(x, \cdot)=0$ in $\Omega \setminus B(x, r)$ for any $r>0$ and thus by \cite[Theorem~1.8]{DEK18}, we see that $G(x,\cdot)$ is continuous away from $x$. On the other hand, it is clear from \eqref{eq0530af} that $G(\cdot, y)$ is continuous away from $y$.

Next, it should be mentioned that Bauman \cite{Bauman85} proved an estimate for the normalized Green's function $\tilde G(x,y)=G(x,y)/G(x_0,y)$:
\begin{equation}
                \label{eq1.54}
\tilde G(x,y)\simeq \int_{\abs{x-y}}^1 \, \frac{r}{w(B(y,r))}\,dr,
\end{equation}
where $x$, $y\in B'$, $\abs{x-y} \in (0,\frac12)$, $x_0\in B\setminus \overline{B'}$, $B'\Subset B$ are open balls,
\[
w(E)=\int_E G(x_0,y)\,dy,
\]
and $G$ is the Green's function in $B$.
Such result was obtained under the condition that the coefficients are bounded and measurable in two dimensions, and uniformly continuous in higher dimensions.
Her proof uses the maximum principle. When the coefficients are of Dini mean oscillation, by using the Harnack type inequality \cite[Lemma 4.2]{DEK18}, we see that $G(x_0,y)\simeq 1$ for any $y$ away from $x_0$ and the boundary $\partial B$, which together with \eqref{eq1.54} implies that
\[
G(x,y) \simeq \log \frac{1}{\abs{x-y}},\quad x \neq y \in B,\quad \abs{x-y}<1/2.
\]
In view of the comparison principle, this result also holds for any bounded smooth domain.
Compared to {\cite{Bauman85}}, we prove a stronger result $G(\cdot, y) \in \mathsf{BMO}(\Omega)$, which gives more information of the Green's function and also implies {\eqref{eq0529af}}.
We also note that our proof also works for elliptic systems satisfying the strong ellipticity condition or the Legendre-Hadamard condition.
More precisely, consider an elliptic system
\[
\sum_{j=1}^N L_{ij} u^j :=\sum_{j=1}^N \sum_{\alpha, \beta=1}^2 A^{\alpha\beta}_{ij} D_{\alpha \beta} u^j - \lambda u^i,\quad i=1,\ldots, N.
\]
Suppose the coefficients $\mathbf{A}=(A^{\alpha\beta}_{ij})$ are bounded, satisfies the Legendre-Hadamard condition
\[
\sum_{\alpha,\beta=1}^2 \sum_{i,j=1}^N A^{\alpha\beta}_{ij} \xi_\alpha\xi_\beta \eta^i \eta^j \ge \nu \abs{\xi}^2 \abs{\eta}^2.
\]
Assume that  $\lambda$ is large enough to guarantee the solvability of Dirichlet problem
\[
L_{ij} u^j = f^i\;\text{ in }\;\Omega,\quad u^i=0\;\text{ on }\;\partial\Omega.
\]
See, for instance, \cite[Theorem 8]{DK11}.
Then, under hypothesis of Theorem~\ref{thm1} (resp. Theorem~\ref{thm3}), there exists an $N \times N$ Green's matrix $\vec G(x,y)=(G_{ij}(x,y))$ that satisfy the conclusions of Theorem~\ref{thm1} (resp. Theorem~\ref{thm3}).
The proof requires only routine adjustment and is omitted.

Finally, we remark that in Theorems~{\ref{thm1}} and {\ref{thm3}}, the dependence of the constant $C$ on $\Omega$ is through the constant $N_0$ in Theorem~{\ref{thm2}}, the constant that is hidden in $H^1(\Omega)$ and $\mathsf{BMO}(\Omega)$ duality relation {\eqref{h1-bmo}}, and the constant that appears in the $W^{2,p}$ estimates, which in turn may depend on the $C^{1,1}$ characteristics of the boundary flattening mappings and the diameter of $\Omega$.
The constant $C$ in {\eqref{eq0531af}} of Theorem~{\ref{thm3}} depends additionally on the $C^{2,\alpha}$ bounds on the boundary flattening mappings.
As the proof of Theorem~{\ref{thm2}} reveals, the constant $N_0$ depends on $\Omega$ through the $W^{2,p}$ estimates, the diameter of $\Omega$, and the $C^{2,\alpha}$ characteristics of $\partial\Omega$.
However, the $C^{2,\alpha}$ regularity of the domain $\Omega$ is  used only for the weak type-$(1,1)$ estimate {\eqref{eq1544w}}, which can be dispensed with if we assume the $L^2$ mean oscillation of $\mathbf A$ satisfies the Dini condition.
In fact, in Theorem~{\ref{thm1}}, where $L^2$ Dini mean oscillation condition is imposed on $\mathbf A$, the $C^{1,1}$ regularity condition on $\Omega$ can be relaxed further.
This condition is essentially used only for the $W^{2,p}$-solvability of \eqref{eq10.36} and the $L^p$-solvability of \eqref{eq20.40th}.
The $L^p$-solvability of \eqref{eq20.40th} follows from the $W^{2,p}$-solvability of \eqref{eq10.36} by using the duality argument in \cite{Esc2000}.
For the $W^{2,p}$-solvability of \eqref{eq10.36} in a bounded domain $\Omega\subset \bR^n$, we only require $\Omega$ to be in $C^{1,\alpha}$, where $\alpha>1-1/\max\{n,p\}$. Thus, in our case by taking $p=2$, it suffices to assume $\Omega$ to be in $C^{1,\alpha}$, where $\alpha>1/2$. We cannot find an explicit reference for this result, so we sketch a proof in Appendix.
It should be mentioned that this argument uses Alexandrov maximum principle and is not applicable to elliptic systems, that is, $C^{1,1}$ regularity requirement on $\Omega$ in Theorem~{\ref{thm1}} cannot be lifted for elliptic systems.
It is also worth noting that in the two dimensional case, the $W^{2,p}$-solvability of \eqref{eq10.36} is also available when $p$ is close to $2$ and $\Omega$ is a bounded convex domain; see, for instance, \cite{Ca67}.
Thus, Theorem~\ref{thm1} also hold when $\Omega$ is bounded and convex.

\section{Notation}

For $x_0\in \bR^n$ and $r>0$, we denote by $B(x_0,r)$ the Euclidean ball with radius $r$ centered at $x_0$, and denote
\[
\Omega(x_0,r):=\Omega \cap B(x_0,r).
\]

We define $\mathsf{BMO}(\Omega)$ as the set of functions $u$ such that
\[
\norm{u}_\ast:=\sup \Set{\fint_{\Omega(x_0,r)} \abs{u-\bar u_{x_0,r}}: x_0\in \overline \Omega,\;\;r>0}
\]
is finite, where we set
\begin{equation}					\label{eq1244th}
\bar u_{x_0,r}:=\begin{cases}
			0 & \text{if } r \ge \dist(x_0, \partial \Omega)\\
			\fint_{\Omega(x_0,r)} u & \text{if }r < \dist(x_0, \partial\Omega).
  \end{cases}
\end{equation}
We shall say that a bounded measurable function $a$ is an atom for $\Omega$ if $a$ is supported in $\Omega(x_0,r)$ for some $x_0 \in \overline\Omega$ and $r>0$ and satisfies
\begin{equation*}
\norm{a}_{L^\infty(\Omega)} \le \frac{1}{\abs{\Omega(x_0,r)}}\quad\text{and}\quad \bar a_{x_0,r} =0.
\end{equation*}
Notice that the latter condition requires $\fint_{\Omega(x_0,r)} a =0$ only if $r <\dist(x_0, \partial\Omega)$.
A function $f$ is in the atomic Hardy space $H^1(\Omega)$ if there is a sequence of atoms $\set{a_i}_{i=1}^\infty$ and a sequence of real numbers $\set{\lambda_i} \in \ell^1$ so that $f=\sum_{i=1}^\infty \lambda_i a_i$.
We define the norm on this space by
\[
\norm{f}_{H^1(\Omega)}=\inf \Set{\sum_{i=1}^\infty\, \abs{\lambda_i}: f=\sum_{i=1}^\infty \lambda_i a_i}.
\]
We note that the expression
\begin{equation}				\label{h1-bmo}
\sup \Set{\int_\Omega a u \,dy: a\text{ is an atom for }\Omega}
\end{equation}
gives an equivalent norm on $\mathsf{BMO}(\Omega)$ and that $\mathsf{BMO}(\Omega)$ may be identified with the dual of the atomic Hardy space $H^1(\Omega)$.
It may seem that our definition of $\mathsf{BMO}(\Omega)$ differs from those in other literatures \cite{Chang94, Jones80, Miyachi90, TKB13} but since we assume that $\Omega$ is at least a $C^{1,\alpha}$ or a bounded convex domain, it is the same.
In particular, $\mathsf{BMO}(\Omega)$ can be identified with the subspace
$\set{f \in \mathsf{BMO}(\bR^2): \supp f \subset \overline \Omega}$, with equivalent norms.

We say that $\Omega$ is a $C^{1,\alpha}$ (resp. $C^{2,\alpha}$) domain if each point on $\partial\Omega$ has a neighborhood in which $\partial\Omega$ is the graph of a $C^{1,\alpha}$ (resp. $C^{2,\alpha}$) function for some $\alpha \in (0,1]$.

\section{Proof of Theorems~\ref{thm1} and \ref{thm3}}			%\label{sec3}
In dimension three and higher, the strategy of \cite{HK07} was used in \cite{HK20} to construct Green's function but it does not work in two dimensions.
Here we follow a duality argument in \cite{TKB13}.

For $y \in \Omega$ and $\epsilon>0$, let $v=G_\epsilon(\cdot,y) \in W^{2,2}(\Omega) \cap W^{1,2}_0(\Omega)$ be a unique strong solution of the problem
\begin{equation}			\label{eq1655th}
Lv=\frac{1}{\abs{\Omega(y,\epsilon)}} \, \chi_{\Omega(y,\epsilon)}\;\text{ in }\;\Omega,\quad v=0\;\text{ on }\;\partial\Omega.
\end{equation}
Since $\mathbf A$ is uniformly continuous in $\Omega$ with its modulus of continuity controlled by $\omega_{\mathbf A}$ (see \cite[Appendix]{HK20}), the unique solvability of the problem \eqref{eq1655th} is a consequence of standard $L^{p}$ theory.
Next, for an $H^1$ atom $a$ in $\Omega$, consider the adjoint problem
\[
L^\ast u = a\;\text{ in }\;\Omega,\quad u=0 \;\text{ on }\;\partial \Omega.
\]
By \cite[Lemma~2]{EM2016}, there exists a unique adjoint solution $u$ in $L^2(\Omega)$, and we have
\[
\fint_{\Omega(y,\epsilon)} u = \int_\Omega a G_\epsilon(\cdot, y).
\]
Then, by Theorem~\ref{thm2}, we have
\[
\Abs{\int_\Omega G_\epsilon(x,y) a(x)\,dx} \le N_0.
\]
Therefore, by the $H^1$ and $\mathsf{BMO}$ duality, we find that the BMO norm of $G_\epsilon(\cdot, y)$ is uniformly bounded, i.e.,
\begin{equation}			\label{bmo_bound}
\norm{G_\epsilon(\cdot,y)}_{*}\le C_0
\end{equation}
for some constant $C_0$ depending only on  $\nu$, $\Omega$,  and $\omega_{\mathbf A}$.
The Banach-Alaoglu theorem gives that for each $y$, there is a sequence $\set{\epsilon_j}$ with $\lim_{j\to \infty} \epsilon_j=0$ and a function $G(\cdot, y) \in \mathsf{BMO}(\Omega)$ so that $G_{\epsilon_j}(\cdot, y)$ converges to $G(\cdot, y)$ in the weak-$\ast$ topology of $\mathsf{BMO}(\Omega)$.
Let us fix a $f \in L^p(\Omega)$ with $p>1$ and let $u$ be the unique adjoint solution of \eqref{eq20.40th}.
Then we have
\begin{equation}				\label{eq0141tu}
\fint_{\Omega(y,\epsilon)} u(x)\,dx = \int_\Omega  G_\epsilon(x, y) f(x)\,dx.
\end{equation}
Since $u$ is continuous in $\overline \Omega$ by \cite[Theorem~1.8]{DEK18}, the left- hand side of \eqref{eq0141tu} converges to $u(y)$.
Since $L^p(\Omega) \subset H^1(\Omega)$ for all $p>1$, we obtain the representation \eqref{eq1747m}.
This gives us that $G(\cdot, y)$ is the Green's function with pole at $y$.
Therefore, by using \eqref{bmo_bound}, we obtain \eqref{eq0528af}.

If we choose any sequence  $\set{\epsilon_k}$ with $\lim_{k\to \infty} \epsilon_k=0$, the above argument gives a subsequence of $G_{\epsilon_k}(\cdot, y)$ which converges to a Green's function.
As the Green's function is unique, the limit must be the function $G(\cdot, y)$.
This implies that the entire family $\set{G_\epsilon(\cdot, y)}_\epsilon$ converges to $G(\cdot, y)$ in the weak-$\ast$ topology of $\mathsf{BMO}(\Omega)$.
By the $W^{2,p}$ estimate for non-divergence form elliptic equations with continuous coefficients, we see from \eqref{eq1655th} that $G_\epsilon(\cdot, y)$ satisfies
\[
\norm{G_\epsilon(\cdot,y)}_{W^{2,2}(\Omega\setminus \Omega(y, r))} \le C(r),\quad r>2\epsilon.
\]
This estimate will also hold for the limit and we see that $G(\cdot,y) \in W^{2,2}(\Omega\setminus B(y, r))$ and $L G(\cdot, y)=0$ in $\Omega\setminus B(y, r)$ for any $r>0$. Moreover, we have $G(\cdot, y)=0$ on $\partial\Omega$.

Next, we show the pointwise bound \eqref{eq0529af}.
For $x_0 \neq y \in \Omega$, we set \[r:=\tfrac12 \,\abs{x_0-y}.\]
In the case when $r\ge \dist(x_0, \partial\Omega)$, we can find a point $\hat x_0 \in \partial\Omega$ such that
\[
\Omega(x_0, r) \subset \Omega(\hat x_0, 2r) \subset \Omega\setminus \set{y}.
\]
Since $LG(\cdot,y) =0$ in $\Omega(\hat x_0,2r)$ and $G(\cdot, y)$ vanishes on $\partial\Omega \cap B(\hat x_0, 2r)$, by the local $W^{2,p}$ estimate, the Sobolev embedding theorem, and a standard iteration argument, we see that the local $L^\infty$ estimate is available for $G(\cdot,y)$, i.e.,
\begin{equation}				\label{eq0804tu}
\abs{G(x_0,y)} \lesssim \fint_{\Omega(\hat x_0, 2r)} \abs{G(x,y)}\,dx.
\end{equation}
Then by \eqref{eq0528af}, we have
\[
\abs{G(x_0,y)} \lesssim \fint_{\Omega(\hat x_0, 2r)} \abs{G(x,y)}\,dx \lesssim \norm{G(\cdot, y)}_{\ast} \lesssim 1,
\]
which clearly yields the bound \eqref{eq0529af}.
In the case when $r < \dist(x_0, \partial\Omega)$, consider a chain of domains $\Omega_j=\Omega(x_0, 2^j r)$ for $j=0,\ldots, N$ so that $2^N r \ge \diam \Omega$.
Notice that $N$ can be chosen so that
\begin{equation}			\label{eq0649tu}
N \lesssim \log \frac{\diam \Omega}{\abs{x_0-y}} +1.
\end{equation}
Since $G(\cdot, y)$ is in $\mathsf{BMO}(\Omega)$, we have
\[
\Abs{\fint_{\Omega_j} G(x,y)\,dx - \fint_{\Omega_{j+1}} G(x,y)\,dx} \lesssim \norm{G(\cdot, y)}_{\ast} \lesssim 1
\]
and by the choice of $N$, we have
\[
\fint_{\Omega_N} \abs{G(x,y)} \,dx=\fint_{\Omega} \abs{G(x,y)} \,dx  \lesssim \norm{G(\cdot, y)}_{\ast} \lesssim 1.
\]
Since $L(G(\cdot,y)-c)=LG(\cdot,y) =0$ in $\Omega_0=B(x_0,r) \subset \Omega$ for any $c \in \bR$, similar to \eqref{eq0804tu},  we have the local $L^\infty$ estimate
\begin{equation}				\label{eq0805tu}
\abs{G(x_0,y)-c} \lesssim \fint_{\Omega_0} \abs{G(x,y)-c}\,dx,\quad \forall c \in \bR.
\end{equation}
Then by taking $c=\fint_{\Omega_0} G(x,y)\,dx$ and using \eqref{eq0528af}, we have
\[
\Abs{G(x_0,y)-\fint_{\Omega_0} G(x,y)\,dx} \lesssim \norm{G(\cdot, y)}_{\ast} \lesssim 1.
\]
Therefore, by telescoping and using \eqref{eq0649tu}, we obtain the bound \eqref{eq0529af}.

Now, we turn to the gradient estimate \eqref{eq0530af}.
In the case when $r\ge \dist(x_0, \partial\Omega)$, similar to \eqref{eq0804tu}, we have the local gradient bound
\[
\abs{D_x G(x_0, y)} \lesssim \frac{1}{r} \fint_{\Omega(\hat x_0,2r)} \abs{G(x,y)}\,dx
\]
and in the case when $r < \dist(x_0, \partial\Omega)$, similar to \eqref{eq0805tu} we have
\[
\abs{D_x G(x_0, y)} \lesssim \frac{1}{r} \fint_{\Omega(x_0,r)} \abs{G(x,y)-c}\,dx,\quad \forall c\in \bR.
\]
In both cases, we get \eqref{eq0530af}.

Finally, we prove the second derivative estimate \eqref{eq0531af}.
In the case when $r< \dist(x_0, \partial\Omega)$, we use the interior $C^2$ estimate (see \cite[Theorem~1.6]{DK17}), $W^{2,p}$ estimate for elliptic equations with continuous coefficients, and a standard iteration argument to get
\[
\abs{D_x^2 G(x_0, y)} \lesssim \frac{1}{r^2} \fint_{\Omega(x_0,r)} \abs{G(x,y)-c}\,dx,\quad \forall c\in \bR.
\]
In the case when $r \ge \dist(x_0, \partial\Omega)$ and $\Omega$ is a $C^{2,\alpha}$ domain, we apply the $C^2$ estimate near the boundary in \cite{DEK18} (see Lemma 2.18 there), the boundary $W^{2,p}$ estimate for elliptic equations with continuous coefficients, and a standard iteration argument to get
\[
\abs{D_x^2 G(x_0, y)} \lesssim \frac{1}{r^2} \fint_{\Omega(\hat x_0,2r)} \abs{G(x,y)}\,dx.
\]
Therefore, we get \eqref{eq408} and \eqref{eq0531af}.

Finally, we prove that $G^\ast(x,y):=G(y,x)$ becomes the Green's function for the adjoint operator $L^\ast$.
For $x_0 \neq y \in \Omega$ and $\rho>0$, let $u=G^\ast_\rho(\cdot,x_0) \in L^{2}(\Omega)$ be a unique solution of the adjoint problem
\[
L^\ast u=\frac{1}{\abs{\Omega(x_0,\rho)}} \chi_{\Omega(x_0,\rho)}\;\text{ in }\;\Omega,\quad u=0\;\text{ on }\;\partial\Omega.
\]
Notice that we have
\[
\fint_{\Omega(y,\epsilon)} G^\ast_\rho(x,x_0)\,dx= \fint_{\Omega(x_0, \rho)} G_\epsilon(x,y)\,dx.
\]
By \cite[Theorem~1.8]{DEK18}, we know that $G^\ast_\rho(\cdot,x_0)$ is continuous in $\overline \Omega$.
Since $G_\epsilon(\cdot, y) \to G(\cdot, y)$ in weak-$\ast$ topology of $\mathsf{BMO}(\Omega)$, by taking the limit $\epsilon \to 0$, we have
\[
G^\ast_\rho(y,x_0)=  \fint_{\Omega(x_0, \rho)} G(x,y)\,dx.
\]
Therefore, we find
\begin{equation}				\label{eq1014sun}
\lim_{\rho \to 0}\, G^\ast_\rho(y, x_0)=G(x_0, y).
\end{equation}
We note that argument around \eqref{eq0804tu} -- \eqref{eq0805tu} shows that
\begin{equation}				\label{eq1015sun}
\abs{G^\ast_\rho(y, x_0)}
=\Abs{\fint_{\Omega(x_0, \rho)} G(x,y)\,dx}
\le C\left( 1+ \log \frac{\diam \Omega}{\abs{x_0-y}} \right),\quad \forall \rho>0.
\end{equation}
For $f\in L^{q}(\Omega)$ with $q\in (1,\infty)$, let $v \in W^{2,q}(\Omega)\cap W^{1,q}_0(\Omega)$ be the strong solution of
\[
Lv=f \;\text{ in }\;\Omega,\quad v=0\;\text{ on }\;\partial \Omega.
\]
Then, we have
\[
\fint_{\Omega(x_0,\rho)} v(x) \,dx = \int_\Omega  G^\ast_\rho(x,x_0) f(x)\,dx,
\]
and thus, by taking the limit, we also get
\[
v(x_0)=\lim_{\rho \to 0} \int_\Omega G^\ast_\rho(x,x_0)f(x)\,dx.
\]
Therefore, by \eqref{eq1014sun}, \eqref{eq1015sun}, and Lebesgue dominated convergence theorem, we obtain the representation formula \eqref{eq1748m}, which means $G^\ast(x,y)=G(y,x)$ is the Green's function for $L^\ast$.
\hfill\qedsymbol

\section{Proof of Theorem~\ref{thm2}}			%\label{sec4}
We first consider the case when the $L^2$ mean oscillation \eqref{13.24g} of $\mathbf A$ satisfies the Dini condition and the domain is $C^{1,1}$, which is assumed in Theorem~\ref{thm1}.
The other case will be treated at the end of the proof.

Suppose that $a$ is supported in $\Omega(y_0,R)$ so that $\norm{a}_{L_\infty} \lesssim 1/R^2$.
For $x_0 \in \Omega$ and $r>0$, we define
\[
u_{x_0,r}=\begin{cases}
			0 & \text{if } r >2\dist(x_0, \partial \Omega)\\
			\fint_{\Omega(x_0,r)} u & \text{if }r \le 2\dist(x_0, \partial\Omega)
  \end{cases}
\]
and set
\begin{equation}					\label{eq0748w}
\phi(x_0, r):= \fint_{\Omega(x_0,r)} \abs{u-u_{x_0,r}}.
\end{equation}
Note that $u_{x_0, r}$ in the above is slightly different from $\bar u_{x_0, r}$ defined in \eqref{eq1244th}.
Let us define the adjoint operator $L_0^\ast$ with constant coefficients
\[
L_0^\ast u:= \sum_{i,j=1}^2 D_{ij} (\bar a^{ij} u), \quad\text{where }\,\bar a^{ij}=\fint_{\Omega(x_0,r)} a^{ij},
\]
and split $u=v^{(r)}+w^{(r)}+\tilde w^{(r)}$, where $w=w^{(r)}$ is the weak solution of
\[
L_0^\ast w=a\text{ in }\;\Omega(x_0,r),\quad w=0\;\text{ on }\;\partial\Omega(x_0,r),
\]
and $\tilde w=\tilde w^{(r)}$ is the $L^2$ adjoint solution of
\[
L_0^\ast \tilde w=\sum_{i,j=1}^2 D_{ij} \left((\bar a^{ij}-a^{ij}) u \chi_{\Omega(x_0,r)} \right)\; \text{ in }\;\Omega,\quad \tilde w=0\;\text{ on }\;\partial\Omega.
\]
By the $L^2$ theory for adjoint equations, we have\footnote{As remarked in the introduction, here we can relax the assumption that $\Omega$ is a $C^{1,1}$ domain to an assumption that $\Omega$ is a $C^{1,\alpha}$ domain for some $\alpha>\frac12$ or that $\Omega$ is a bounded convex domain.}
\[
\left(\int_{\Omega(x_0,r)} \abs{\tilde w}^2 \right)^{1/2} \le \left(\int_{\Omega} \abs{\tilde w}^2 \right)^{1/2} \lesssim \left(\int_{\Omega(x_0,r)} \abs{\bar A-A}^2  \right)^{1/2} \norm{u}_{L^\infty(\Omega(x_0,r))}.
\]
By H\"older's inequality, we then have
\begin{equation}					\label{eq1019sat}
\fint_{\Omega(x_0,r)} \abs{\tilde w} \lesssim \left(\fint_{\Omega(x_0,r)} \abs{A-\bar A}^2 \right)^{1/2} \norm{u}_{L^\infty(\Omega(x_0,r))}\lesssim \omega_{A}(r) \,\norm{u}_{L^\infty(\Omega(x_0,r))}.
\end{equation}

Now we turn to the decay estimate of $\phi(x_0, r)$.
Let $\kappa \in (0,\frac14)$ be fixed.
Note that $v=v^{(r)}=u-w^{(r)}-\tilde w^{(r)}$ satisfies
\[
L_0^\ast v=0\; \text{ in }\;\Omega(x_0,r),\quad v=u=0 \text{ in }\; \partial\Omega \cap B(x_0,r).
\]
By an interior and boundary estimate for elliptic equations with constant coefficients, we have (recall $v_{x_0,r}=0$ if $B(x_0,r/2)$ intersects $\partial\Omega$)
\begin{equation}				\label{eq13.46f}
\fint_{\Omega(x_0, \kappa r)} \abs{v - v_{x_0, \kappa r}}  \le 2\kappa r \norm{Dv}_{L^\infty(\Omega(x_0,r/3))} \le C_0 \kappa \fint_{\Omega(x_0,r)} \abs{v-v_{x_0,r}}.
\end{equation}
Here, $C_0$ is a constant depending only on $\nu$ and $\Omega$.
By using the decomposition $u=v+w+\tilde w$ and \eqref{eq13.46f}, we obtain
\begin{align*}
\fint_{\Omega(x_0, \kappa r)} \abs{u - u_{x_0, \kappa r}}
& \le \fint_{\Omega(x_0, \kappa r)} \abs{v - v_{x_0, \kappa r}} + 2\fint_{\Omega(x_0, \kappa r)} \abs{w} +2\fint_{\Omega(x_0, \kappa r)} \abs{\tilde w} \\
& \le  C_0\kappa \fint_{\Omega(x_0,r)} \abs{u - u_{x_0,r}} +C (\kappa^{-2}+1) \fint_{\Omega(x_0,r)} \left(\abs{w} + \abs{\tilde w} \right).
\end{align*}
Here, we used the obvious facts that
\[
u_{x_0, r}=v_{x_0,r} +  w_{x_0,r} + \tilde w_{x_0,r},\quad
\abs{w_{x_0,r}} \le \fint_{\Omega(x_0,r)} \abs{w},\quad \abs{\tilde w_{x_0,r}} \le \fint_{\Omega(x_0,r)} \abs{\tilde w}.
\]
Therefore, by \eqref{eq1019sat}, we have
\[
\phi(x_0,\kappa r) \le C_0 \kappa  \phi(x_0, r) +  C (\kappa^{-2}+1) \omega_{\mathbf A}(r) \,\norm{u}_{L^\infty(\Omega(x_0,r))}
+ C (\kappa^{-2}+1) \fint_{\Omega(x_0,r)}\abs{w}.
\]
Now we fix a $\kappa=\kappa(\nu, \Omega) \in (0,\frac14)$ sufficiently small so that $C_0\kappa \le 1/2$.
Then, we obtain
\begin{equation}				\label{eq16.57m}
\phi(x_0, \kappa r) \le \frac12 \phi(x_0, r)+ C \omega_{\mathbf A}(r) \,\norm{u}_{L^\infty(\Omega(x_0,r))}+  C \fint_{\Omega(x_0,r)}\abs{w^{(r)}}.
\end{equation}
By iterating, for $j=1,2,\ldots$, we get
\begin{equation}				\label{eq1525mon}
\phi(x_0, \kappa^j r) \le 2^{-j} \phi(x_0, r) +C \norm{u}_{L^\infty(\Omega(x_0,r))} \sum_{i=1}^{j} 2^{-i} \omega_{\mathbf A}(\kappa^{j-i} r) + C\psi_j(x_0,r),
\end{equation}
where we set
\begin{equation}				\label{eq10.55}
\psi_j(x_0,t):=\sum_{i=1}^{j} 2^{-i}\fint_{\Omega(x_0,\kappa^{j-i}r)}\abs{w^{(\kappa^{j-i}r)}}.
\end{equation}
Note that
\begin{align}
							\nonumber
\sum_{j=1}^\infty \sum_{i=1}^j 2^{-i} \omega_{\mathbf A}(\kappa^{j-i} r)& =\sum_{i=1}^\infty \sum_{j=i}^\infty 2^{-i} \omega_{\mathbf A}(\kappa^{j-i} r)
=\sum_{i=1}^\infty  2^{-i}\sum_{j=0}^\infty \omega_{\mathbf A}(\kappa^{j} r)\\
							\label{eq0752w}
&=\sum_{j=0}^\infty \omega_{\mathbf A}(\kappa^{j} r) \lesssim \int_0^r \frac{\omega_{\mathbf A}(t)}{t}\,dt.
\end{align}

\begin{lemma}					\label{lem-key}
We have
\begin{equation}				\label{eq1855w}
\sum_{j=1}^\infty \psi_j(x_0,r) \lesssim 1,\quad \forall x_0 \in \Omega, \quad 0<\forall r <\diam \Omega.
\end{equation}
\end{lemma}
Take the lemma for granted now.
Then from \eqref{eq1525mon} and \eqref{eq0752w} we find that
\begin{equation}				\label{eq0750w}
\sum_{j=0}^\infty \phi(x_0, \kappa^j r) \lesssim \phi(x_0, r) +  \norm{u}_{L^\infty(\Omega(x_0,r))} \int_0^r \frac{\omega_{\mathbf A}(t)}{t}\,dt +1.
\end{equation}
Since we have
\[
\abs{u_{x_0,\kappa r} - u_{x_0, r}} \le
\abs{u(x)-u_{x_0,r}} + \abs{u(x) - u_{x_0,\kappa r}},
\]
by taking average over $x \in \Omega(x_0, \kappa r)$, we obtain
\[
\abs{u_{x_0,\kappa r} - u_{x_0,r}} \le \phi(x_0,\kappa r) + \phi(x_0, r).
\]
Then, by iterating, we get
\begin{equation}				\label{eq0756w}
\abs{u_{x_0,\kappa^i r} - u_{x_0, r}} \le 2 \sum_{j=0}^i \phi(x_0, \kappa^j r).
\end{equation}
By Theorem~1.8 in \cite{DEK18}, we find that $u$ is continuous in $\overline\Omega$, and thus we see that
\begin{equation}				\label{eq1402t}
\lim_{i\to \infty} u_{x_0,\kappa^i r}=u(x_0).
\end{equation}
Therefore, by taking $i\to \infty$ in \eqref{eq0756w} and using \eqref{eq0750w}, we get
\begin{equation}				\label{eq13.13}
\abs{u(x_0)-u_{x_0,r}} \lesssim \phi(x_0, r)+
\norm{u}_{L^\infty(\Omega(x_0,r))} \int_0^r \frac{\omega_{\mathbf A}(t)}t \,dt+1,
\end{equation}
which implies that
\begin{equation}			\label{eq14.39t}
\abs{u(x_0)}  \lesssim r^{-2} \norm{u}_{L^1(\Omega(x_0,r))} + \norm{u}_{L^\infty(\Omega(x_0,r))} \int_0^r \frac{\omega_{\mathbf A}(t)}t \,dt+1.
\end{equation}
Now, taking the supremum for $x_0\in \Omega(x,r)$, where $x \in \overline\Omega$, we have
\[
\norm{u}_{L^\infty(\Omega(x,r))} \le
C \left( r^{-2} \norm{u}_{L^1(\Omega(x, 2r))} + \norm{u}_{L^\infty(\Omega(x, 2r))} \int_0^r \frac{\omega_{\mathbf A}(t)}t \,dt +1\right).
\]
We fix $r_0<\frac13$ such that for any $0<r\le r_0$,
\[
C \int_0^{r} \frac{\omega_{\mathbf A}(t)}t \,dt \le \frac1{3^2}.
\]
Then, we have for any $x \in \Omega$ and $0<r\le r_0$ that
\[
\norm{u}_{L^\infty(\Omega(x,r))} \le
3^{-2}\norm{u}_{L^\infty(\Omega(x, 2r))} + C r^{-2} \norm{u}_{L^1(\Omega(x, 2r))} + C.
\]
For $k=1,2,\ldots$, denote $r_k=3-2^{1-k}$.
Note that $r_{k+1}-r_k=2^{-k}$ for $k\ge 1$ and $r_1=2$.
Without loss of generality, let us assume that $0 \in \Omega$.
For $x\in \Omega_{r_k}=\Omega(0, r_k)$ and $r\le 2^{-k-2}$, we have $B(x, 2r) \subset B_{r_{k+1}}$. We take $k_0\ge 1$ sufficiently large such that $2^{-k_0-2}\le r_0$.
It then follows that for any $k\ge k_0$,
\[
\norm{u}_{L^\infty(\Omega_{r_k})} \le 3^{-2} \norm{u}_{L^\infty(\Omega_{r_{k+1}})}+C 2^{2k} \norm{u}_{L^1(\Omega_3)} + C.
\]
By multiplying the above by $3^{-2k}$ and then summing over $k=k_0, k_0+1,\ldots$, we reach
\[
\sum_{k=k_0}^\infty 3^{-2k}\norm{u}_{L^\infty(\Omega_{r_k})} \le \sum_{k=k_0}^\infty 3^{-2(k+1)} \norm{u}_{L^\infty(\Omega_{r_{k+1}})}+C \norm{u}_{L^1(\Omega_3)} + C.
\]
Since we assume that $u\in L^\infty(\Omega)$, the summations on both sides are convergent.
Therefore, we have
\[
3^{-2k_0}\norm{u}_{L^\infty(\Omega_{r_{k_0}})} \le C  \norm{u}_{L^1(\Omega_3)} + C,
\]
which in turn implies (recall $2 < r_{k_0}$)
\begin{equation}					\label{eq10.23m}
\norm{u}_{L^\infty(\Omega_2)} \lesssim  \norm{u}_{L^1(\Omega_3)} + 1.
\end{equation}

\begin{lemma}						\label{lem-global}
We have
\[
\norm{u}_{L^1(\Omega)} \lesssim  1.
\]
\end{lemma}
\begin{proof}
For any $f \in L^\infty(\Omega)$, let $v \in W^{2,2}(\Omega)\cap W^{1,2}_0(\Omega)$ be the unique solution of the problem
\[
Lv=f\;\text{ in }\;\Omega,\quad v=0\;\text{ on }\;\partial\Omega.
\]
See, for instance, \S 11.3 in {\cite{K08}}.
By the Sobolev embedding and the $W^{2,2}$ theory, we have $v \in C(\overline \Omega)$ and
\[
\norm{v}_{L^\infty(\Omega)} \le C\norm{v}_{W^{2,2}(\Omega)}
\le C \norm{f}_{L^2(\Omega)} \le C \norm{f}_{L^\infty(\Omega)},
\]
where $C=C(\nu, \omega_{\mathbf A}, \Omega)$.\footnote{Here, we only need that  $\Omega$ is a bounded $C^{1,1}$ domain. In the scalar case, by using Alexandrov estimate, one can bypass the $W^{2,2}$ estimate and directly get $\norm{v}_{L^\infty(\Omega)} \le C \norm{f}_{L^2(\Omega)} \le C \norm{f}_{L^\infty(\Omega)}$, and thus only boundedness of $\Omega$ is needed.}
Since
\[
\int_\Omega u f \,dx= \int_\Omega a v\,dx,
\]
and $\norm{a}_{L^1(\Omega)} \le  1$, we find
\[
\Abs{\int_\Omega u f\,dx } \le C \norm{f}_{L^\infty(\Omega)},
\]
which implies that $\norm{u}_{L^1(\Omega)} \le C$.
\end{proof}
By Lemma~\ref{lem-global} and \eqref{eq10.23m}, we have (recalling that $0\in \Omega$ is an arbitrary choice)
\[
\norm{u}_{L^\infty(\Omega)} \lesssim 1
\]
as desired.
The proof of the theorem is complete once we prove Lemma~\ref{lem-key}.

\begin{proof}[Proof of Lemma~\ref{lem-key}]
%We now prove the claim \eqref{eq1855w}.
Let $G_r(x,y)$ denote the Green's function for the constant coefficient operator $L_0^\ast =  \sum D_{ij} \bar a^{ij} =  \sum \bar a^{ij} D_{ij}$ in $\Omega(x_0,r)$.
Then we have
\begin{equation*}
w^{(r)}(x)=\int_{\Omega(x_0,r)} G_r(x,y) a(y)\,dy,\quad \forall x \in \Omega(x_0,r).
\end{equation*}
We shall use the following  estimates for Green's function $G_r(x,y)$:
\begin{align}
							\label{green1}
\abs{G_r(x,y)} &\le C\left(1+\log \frac{2r}{\abs{x-y}} \right),\\
							\label{green2}
\abs{G_r(x,y)-G_r(x,y')} &\le C\,\frac{\abs{y-y'}}{\abs{x-y}}.
\end{align}
Since $L_0$ is a constant coefficients operator, the above estimates are widely known.
We remark that \eqref{green2} may not be sharp if $\abs{y-y'} > \frac12 \abs{x-y}$ but it is still a legitimate estimate, which can be seen by telescoping: choose a sequence of points $y_0, y_1, \ldots, y_N$ in $\Omega$ such that\footnote{This is always available when $\Omega$ is, for example, a bounded Lipschitz domain.}
\[
y_0=y,\;\; y_N=y',\;\; \abs{y_i-y_{i-1}} \le \frac12 \abs{x-y},\;\; \abs{x-y} \le \abs{x-y_i},\;\text{ and }\;\sum_{i=1}^N\, \abs{y_i-y_{i-1}} \simeq \abs{y-y'}.
\]
Then, we have
\[
\abs{G(x,y)-G(x,y')} \le \sum_{i=1}^N\, \abs{G(x,y_{i-1})-G(x,y_i)}\lesssim \sum_{i=1}^N \,\frac{\abs{y_{i-1}-y_i}}{\abs{x-y_i}} \lesssim \frac{\abs{y-y'}}{\abs{x-y}},
\]
which establishes \eqref{green2}.

To estimate $w^{(r)}$, we consider two cases: $\abs{x_0-y_0} \le 2R$ and $\abs{x_0-y_0} >2R$.

{\bf Case 1.} Let us first consider the case when $\abs{x_0-y_0} \le 2R$.
If $r \le 3R$, we use the size condition of $a$ and \eqref{green1} to estimate $w^{(r)}(x)$ for $x \in \Omega(x_0,r)$ as follows:
\begin{align}
						\nonumber			
\abs{w^{(r)}(x)} &\le \int_{\Omega(x_0,r)} \abs{G_r(x,y)}\, \abs{a(y)} \,dy \lesssim \frac{1}{R^2} \int_{\Omega(x_0,r) \cap \Omega(y_0,R)} \left(1+\log\frac{2r}{\abs{x-y}}\right)\,dy\\
						\label{eq1229mon}					
&\lesssim \frac{1}{R^2} \int_{B(x,2r)}\left( 1+ \log\frac{2r}{\abs{x-y}}\right)\,dy \lesssim  \frac{r^2}{R^2}.
\end{align}
Therefore, we have
\begin{equation}
					\label{eq1234th}
\fint_{\Omega(x_0,r)} \abs{w^{(r)}} \lesssim \frac{r^2}{R^2}\quad\text{when }\; r\le 3R.
\end{equation}
If $r > 3R$, we have $\Omega(x_0,r) \supset \Omega(y_0,R)$.
In the case when $R <\dist(y_0, \partial\Omega)$, we use the cancellation property of $a$ to find that
\[
w^{(r)}(x) =\int_{\Omega(x_0,r)} G_r(x,y) a(y)\,dy=\int_{\Omega(y_0,R)} \left(G_r(x,y)-G_r(x, y_0)\right) a(y)\,dy.
\]
Then by the estimate \eqref{green2} and using the symmetry, we have
\begin{align*}
\abs{w^{(r)}(x)} &\le \int_{\Omega(y_0,R)} \abs{G_r(x,y) -G_r(x,y_0)}\, \abs{a(y)} \,dy \\
& \lesssim \frac{1}{R} \int_{B(y_0, R)} \min \left( \frac{1}{\abs{x-y}}, \frac{1}{\abs{x-y_0}} \right)\,dy.
\end{align*}
In the case when $\abs{x-y_0} \le R$, we estimate
\[
\abs{w^{(r)}(x)} \lesssim \frac{1}{R} \int_{B(y_0, R)}\frac{1}{\abs{x-y}}\,dy \lesssim \frac{1}{R} \int_{B(x,2R)}\frac{1}{\abs{x-y}}\,dy \lesssim 1.
\]
In the case when $\abs{x-y_0} > R$, we estimate
\[
\abs{w^{(r)}(x)} \lesssim \frac{1}{R} \int_{B(y_0, R)}\frac{1}{\abs{x-y_0}}\,dy \lesssim \frac{R}{\abs{x-y_0}}.
\]
Note that for $x\in \Omega(x_0,r)$, we have
\[
\abs{x-y_0} \le \abs{x-x_0}+ \abs{x_0-y_0} \le r+2R < 2r.
\]
Combining these together, we have
\begin{align*}
\fint_{\Omega(x_0,r)} \abs{w^{(r)}(x)}\,dx  &\lesssim \frac{1}{r^2} \left( \int_{B(y_0, R)} \abs{w^{(r)}(x)}\,dx+ \int_{B(y_0, 2r)\setminus B(y_0, R)} \abs{w^{(r)}(x)}\,dx \right) \\
&\lesssim \frac{1}{r^2} \left(\int_{B(y_0, R)} dx+  \int_{B(y_0, 2r)} \frac{R}{\abs{x-y_0}}\,dx\right) \lesssim \frac{R^2}{r^2}+ \frac{R}{r} \lesssim \frac{R}{r}.
\end{align*}
In the case when $R \ge \dist(y_0, \partial\Omega)$, we can find $y' \in \partial \Omega(x_0,r)$ such that $\abs{y'-y_0} \le R$.
Therefore, we have
\[
w^{(r)}(x) =\int_{\Omega(x_0,r)} G_r(x,y) a(y)\,dy=\int_{\Omega(y_0,R)} \left(G_r(x,y)-G_r(x, y')\right) a(y)\,dy.
\]
Notice that $\Omega(y_0, R) \subset \Omega(y', 2R)$.
Then, by repeating the above argument with $y'$ in place of $y_0$, we get the same conclusion.
Therefore, we have
\begin{equation}			\label{eq1235th}
\fint_{\Omega(x_0,r)} \abs{w^{(r)}(x)}\,dx  \lesssim  \frac{R}{r} \quad\text{when }\; r> 3R.
\end{equation}

Now, let us look into $\psi_j(x_0,r)$, which is defined in \eqref{eq10.55}.
Let $\ell$ be the largest integer satisfying $\kappa^\ell r>3R$.
In the case when $\ell <0$, we have $r\le 3R$, and thus by \eqref{eq1234th},
\[
\psi_j(x_0, r)= \sum_{i=1}^{j} 2^{-i} \left(\frac{\kappa^{j-i}r}{R}\right)^2  \lesssim 2^{-j} \frac{r^2}{R^2} \lesssim 2^{-j}.
\]
In the case when $0\le \ell < j$, we have by \eqref{eq1234th} and \eqref{eq1235th} that
\begin{align*}
\psi_j(x_0,r)&= \sum_{i=1}^{j-\ell-1} 2^{-i} \left(\frac{\kappa^{j-i}r}{R}\right)^2
+\sum_{i=j-\ell}^{j} 2^{-i} \frac{R}{\kappa^{j-i}r} \\
&\lesssim \left(\frac{\kappa^\ell r}{R}\right)^2 2^{\ell-j} + \left(\frac{R}{\kappa^\ell r}\right) 2^{\ell-j} \lesssim 2^{\ell-j},
\end{align*}
where we used $\kappa^\ell r \simeq R$, which follows from the choice of $\ell$.
Finally, in the case when $\ell \ge j$, we have by \eqref{eq1235th} that
\[
\psi_j(x_0,r)=R \sum_{i=1}^{j} 2^{-i} \frac{R}{\kappa^{j-i}r}  \lesssim \frac{R}{\kappa^j r} = \frac{R}{\kappa^{\ell} r} \kappa^{\ell-j} \lesssim \kappa^{\ell-j}.
\]
Therefore, we have
\begin{align}
						\nonumber
\sum_{j=0}^\infty \psi_j(x_0,r) &\lesssim \sum_{j=0}^\infty \left( 2^{-j}[\ell<0]+ 2^{\ell-j}[0\le \ell<j]+  \kappa^{\ell-j}[\ell \ge j] \right) \\
						\label{eq1939sun}
& \lesssim \sum_{j=0}^\infty 2^{-j} + \sum_{j=\ell}^\infty 2^{\ell-j}+\sum_{j=0}^{\ell} \kappa^{\ell-j}\lesssim 1.
\end{align}
This completes the proof of \eqref{eq1855w} in the case when $\abs{x_0-y_0} \le 2R$.

{\bf Case 2.} Next, we turn to the proof of \eqref{eq1855w} in the case when $\abs{x_0-y_0}>2R$.
We first consider the case when $\Omega(y_0,R) \not\subset \Omega(x_0,r)$.
Instead of \eqref{eq1229mon}, we estimate
\begin{equation}				\label{eq1109th}
\abs{w^{(r)}(x)} \le \int_{\Omega(x_0,r)} \abs{G_r(x,y)}\, \abs{a(y)} \,dy \lesssim \frac{1}{R^2} \int_{\Omega(x_0,r) \cap \Omega(y_0,R)} \abs{G_r(x,y)}\,dy.
\end{equation}
If $\Omega(x_0,r) \cap \Omega(y_0,R) =\emptyset$, then the above integral is zero.
Therefore, we have
\begin{equation}					\label{eq1048w}
\fint_{\Omega(x_0,r)} \abs{w^{(r)}(x)}\,dx =0\quad\text{when }\;\Omega(x_0,r) \cap \Omega(y_0,R) =\emptyset.
\end{equation}
If $\Omega(x_0,r) \cap \Omega(y_0,R) \neq \emptyset$, then for $y\in \Omega(x_0,r) \cap \Omega(y_0,R)$, we can find $y' \in \partial \Omega(x_0,r)$ such that $\abs{y-y'} \le 2R$   (recall  $\Omega(y_0,R) \not\subset \Omega(x_0,r)$).
Then, by the Green's function estimate \eqref{green2}, we have
\[
\abs{G_r(x,y)}=\abs{G_r(x,y)-G_r(x,y')} \lesssim \frac{R}{\abs{x-y}}.
\]
Hence by \eqref{eq1109th}, we obtain
\[
\abs{w^{(r)}(x)} \lesssim \frac{1}{R} \int_{\Omega(x_0,r) \cap \Omega(y_0,R)} \frac{1}{\abs{x-y}}\,dy.
\]
In the case when $\abs{x-y_0} >2R$, we have $\abs{x-y}>R$, and thus
\[
\abs{w^{(r)}(x)}  \lesssim \frac{1}{R^2} \int_{B(y_0, R)} \,dy \lesssim 1.
\]
In the case when $\abs{x-y_0} \le 2R$, we have $B(y_0, R) \subset B(x, 3R)$, and thus
\[
\abs{w^{(r)}(x)}  \lesssim \frac{1}{R} \int_{B(x, 3R)} \frac{1}{\abs{x-y}}\,dy \lesssim 1.
\]
In both cases, we have
\begin{multline}				\label{eq0751tue}
\fint_{\Omega(x_0,r)} \abs{w^{(r)}(x)}\,dx \lesssim 1\\
\text{when }\;\Omega(y_0,R) \not\subset \Omega(x_0,r)\;\text{ and }\;\Omega(x_0,r) \cap \Omega(y_0,R) \neq\emptyset.
\end{multline}

In the case when $\Omega(y_0,R) \subset \Omega(x_0,r)$, we have $r>3R$ (recall $\abs{x_0-y_0}>2R$) and similar to \eqref{eq1235th}, we have
\begin{equation}				\label{eq1108tue}
\fint_{\Omega(x_0,r)} \abs{w^{(r)}(x)}\,dx \lesssim \frac{R}{r}.
\end{equation}
Since we assume $\abs{x_0-y_0}>2R$, there exists the smallest integer $\ell$ satisfying
\[
\Omega(x_0, \kappa^\ell r) \cap \Omega(y_0,R) =\emptyset.
\]
Then, by the choice of $\ell$,
\[
\Omega(x_0, \kappa^{\ell-1}r) \cap \Omega(y_0,R) \neq \emptyset,
\]
and thus $\kappa^{\ell-1} r >\abs{x_0-y_0}-R>R$.
Also, since  $\kappa < \frac14$, we note that
\[\Omega(y_0,R) \subset \Omega(x_0, 2R+\kappa^{\ell-1}r) \subset \Omega(x_0, \kappa^{\ell-2}r).
\]
In the case when $\ell \le 1$, then by \eqref{eq1048w} and \eqref{eq0751tue}, we have
\[
\psi_j(x_0,r)=2^{-j} \fint_{\Omega(x_0,r)} \abs{w^{(r)}} \lesssim 2^{-j}.
\]
In the case when $1<\ell \le j$, then by \eqref{eq1048w}, \eqref{eq0751tue}, and \eqref{eq1108tue}, we have
\begin{align*}
\psi_j(x_0,r)&= 2^{-(j-\ell+1)} \fint_{\Omega(x_0, \kappa^{\ell-1}r)} \abs{w^{(\kappa^{\ell-1}r)}}+ \sum_{i=j-\ell+2}^{j} 2^{-i} \fint_{\Omega(x_0, \kappa^{j-i}r)} \abs{w^{(\kappa^{j-i}r)}}\\
& \lesssim 2^{\ell-j} + \sum_{i=j-\ell+1}^{j} 2^{-i} \frac{R}{\kappa^{j-i}r} \lesssim 2^{\ell-j}+
\frac{R}{\kappa^{\ell-1} r} 2^{\ell-j} \lesssim 2^{\ell-j}.
\end{align*}
In the case when $\ell > j$, then by \eqref{eq0751tue} we have
\[
\psi_j(x_0,r) = \sum_{i=1}^{j} 2^{-i} \fint_{\Omega(x_0, \kappa^{j-i}r)} \abs{w^{(\kappa^{j-i}r)}} \lesssim \sum_{i=1}^{j} 2^{-i}  \frac{R}{\kappa^{j-i}r}  \lesssim \frac{R}{\kappa^j r} =\frac{R}{\kappa^{\ell-1}r} \kappa^{\ell-1-j} \lesssim \kappa^{\ell-j}.
\]
Therefore, similar to \eqref{eq1939sun}, we obtain
\[
\sum_{j=0}^\infty \psi_j(x_0,r) \lesssim \sum_{j=0}^\infty\left( 2^{-j}[\ell \le 1] + 2^{\ell-j} [1<\ell \le j]+ \kappa^{\ell-j}[\ell>j] \right) \lesssim 1
\]
as desired.
\end{proof}

Now, we consider the other case when and $\mathbf A$ is of Dini mean oscillation and $\Omega$ is a $C^{2,\alpha}$ domain.
We note that $L^2$ mean oscillation \eqref{13.24g} is used to obtain the estimate \eqref{eq1019sat}, which seems no longer available with $L^1$ mean oscillation \eqref{13.24f}.
We shall derive an estimate which substitutes \eqref{eq1019sat} as follows.
Since $L_0^*$ has constant coefficients and $\partial\Omega$ is of $C^{2,\alpha}$, we have weak type-$(1,1)$ estimate\footnote{Here, we use the assumption that $\Omega$ is a bounded $C^{2,\alpha}$ domain for some $\alpha>0$.} (see e.g., \cite[Lemma~2.4]{DEK18})
\begin{equation}				\label{eq1544w}
\Abs{\set{x \in \Omega: \abs{\tilde w(x)} >t}} \lesssim \frac{1}{t} \int_\Omega \Abs{(\bar A-A) u \chi_{\Omega(x_0,r)}} \le \frac{1}{t} \left(\int_{\Omega(x_0, r)} \abs{\bar A-A}\right) \,\norm{u}_{L^\infty(\Omega(x_0,r))},
\end{equation}
which implies that for any $p \in (0,1)$ we have
\begin{align*}
\int_{\Omega(x_0,r)} \abs{\tilde w}^{p}  &=\int_0^\tau+ \int_\tau^\infty  p t^{p-1} \Abs{\set{x \in \Omega(x_0,r): \abs{\tilde w(x)} >t}} \,dt \\
&\lesssim \abs{\Omega(x_0,r)}  \int_0^\tau  p t^{p-1}\,dt + \abs{\Omega(x_0,r)} \,\omega_A(r) \,\norm{u}_{L^\infty(\Omega(x_0,r))} \int_\tau^\infty  p t^{p-2}\,dt \\
&= \abs{\Omega(x_0,r)} \tau^{p} + \frac{p}{1-p}\, \abs{\Omega(x_0,r)} \,\omega_A(r)\,\norm{u}_{L^\infty(\Omega(x_0,r))} \tau^{p-1}.
\end{align*}
By taking  $\tau= \omega_A(r)\,\norm{u}_{L^\infty(\Omega(x_0,r))}$ in the above, we obtain
\[
\left(\fint_{\Omega(x_0,r)} \abs{\tilde w}^{p} \right)^{1/p} \lesssim  \omega_A(r)\,\norm{u}_{L^\infty(\Omega(x_0,r))},
\]
which substitutes \eqref{eq1019sat}.
For the sake of definiteness, we shall take $p=\frac12$ in the above and get
\begin{equation}					\label{eq1019}
\left(\fint_{\Omega(x_0,r)} \abs{\tilde w}^{\frac12} \right)^{2} \lesssim  \omega_A(r)\,\norm{u}_{L^\infty(\Omega(x_0,r))}.
\end{equation}
Also, instead of \eqref{eq13.46f}, we have the following:
First note that for any $\kappa \in (0,\frac14)$, we always have
\[
\left(\fint_{\Omega(x_0, \kappa r)} \abs{v - v_{x_0, \kappa r}}^{\frac12} \right)^2 \le 2\kappa r \norm{Dv}_{L^\infty(\Omega(x_0, \kappa r))} \le 2\kappa r \norm{Dv}_{L^\infty(\Omega(x_0,r/3))},
\]
for if $B(x_0,\kappa r/2)$ intersects $\partial\Omega$ and we have $v-v_{x_0, \kappa r}=v-0=v-v(\bar x)$ for some $\bar x \in \partial\Omega \cap B(x_0, \kappa r/2)$.

Next, in the case when $r \le 2\dist(x_0, \partial \Omega)$, the interior estimates for equations with constant coefficients yield
\[
\norm{Dv}_{L^\infty(\Omega(x_0,r/3))} \lesssim \frac{1}{r}\left( \fint_{\Omega(x_0, r/2)} \abs{v-c}^{\frac12} \right)^2 \lesssim \frac{1}{r} \left( \fint_{\Omega(x_0, r)} \abs{v-c}^{\frac12} \right)^2,\quad \forall c \in \bR,
\]
for $v-c$ satisfies $L_0^* (v-c)=0$ in $\Omega(x_0, r/2) = B(x_0, r/2)$ and $D(v-c)=Dv$.
In the case when $r > 2\dist(x_0, \partial \Omega)$, by the boundary estimate we have
\[
\norm{Dv}_{L^\infty(\Omega(x_0,r/3))} \lesssim \frac{1}{r} \left( \fint_{\Omega(x_0,r)} \abs{v}^{\frac12} \right)^2.
\]
Combining these together, we conclude that
\begin{equation}			\label{eq15.26m}
\left(\fint_{\Omega(x_0, \kappa r)} \abs{v - v_{x_0, \kappa r}}^{\frac12} \right)^2 \le C_0 \kappa \left( \fint_{\Omega(x_0, r)} \abs{v-c}^{\frac12} \right)^2,
\end{equation}
where $c=0$ when $r > 2\dist(x_0, \partial \Omega)$ and $c\in \bR$ is arbitrary otherwise.

We recall the facts that for all $a, b \ge 0$, we have
\[
(a+b)^{\frac12} \le a^{\frac12} + b^{\frac12}, \quad (a+ b)^{2} \le 2(a^2+b^2),
\]
and
\[
\norm{f+g}_{L^{1/2}} \le 2 \left( \norm{f}_{L^{1/2}} + \norm{g}_{L^{1/2}}\right).
\]

By using the decomposition $u=v+w+\tilde w$, the above facts, and \eqref{eq15.26m}, we obtain
\begin{align*}
\left(\fint_{\Omega(x_0, \kappa r)} \abs{u - v_{x_0, \kappa r}}^{\frac12}\right)^2
& \le2 \left(\fint_{\Omega(x_0, \kappa r)} \abs{v - v_{x_0, \kappa r}}^{\frac12}\right)^2 + 2\left(\fint_{\Omega(x_0, \kappa r)} \abs{w +\tilde w}^{\frac12} \right)^2  \\
& \le2 C_0 \kappa \left(\fint_{\Omega(x_0, r)} \abs{v-c}^{\frac12}\right)^2 +2 \kappa^{-4} \left(\fint_{\Omega(x_0,r)} \abs{w +\tilde w}^{\frac12} \right)^2 \\
& \le  4C_0\kappa \left(\fint_{\Omega(x_0, r)} \abs{u-c}^{\frac12}\right)^2 + (4C_0 \kappa+2\kappa^{-4}) \left(\fint_{\Omega(x_0,r)} \abs{w +\tilde w}^{\frac12} \right)^2\\
& \le 4C_0\kappa \left(\fint_{\Omega(x_0, r)} \abs{u-c}^{\frac12}\right)^2 +(C_0+2\kappa^{-4})\,\omega_A(r)\,\norm{u}_{L^\infty(\Omega(x_0,r))}\\	
&\qquad\qquad+ (C_0+2\kappa^{-4} )\fint_{\Omega(x_0, \kappa r)} \abs{w},
\end{align*}
where we used \eqref{eq1019} and H\"older's inequality in the last step.

Therefore, in place of \eqref{eq0748w}, if we set
\begin{equation*}%					\label{eq0748m}
\phi(x_0, r):=\left\{
\begin{aligned}
\inf_{c\in \bR}\left(\fint_{\Omega(x_0,r)} \abs{u-c}^{\frac12} \right)^{2}\quad & \text{if } r \le 2\dist(x_0, \partial \Omega)\\
\left(\fint_{\Omega(x_0,r)} \abs{u}^{\frac12} \right)^{2}\quad & \text{if } r >2\dist(x_0, \partial \Omega),\\
\end{aligned}
\right.\end{equation*}
then we still obtain \eqref{eq16.57m} and \eqref{eq0750w}.
Moreover, if we set $q_{x_0,r} \in \bR$ to be a number such that
\[
\phi(x_0,r)=\left(\fint_{\Omega(x_0,r)} \abs{u-q_{x_0,r}}^{\frac12} \right)^{2},
\]
then instead of \eqref{eq0756w}, \eqref{eq1402t}, and \eqref{eq13.13}, we have
\begin{gather*}
\abs{q_{x_0,\kappa^i r} - q_{x_0, r}} \le 4 \sum_{j=0}^i \phi(x_0, \kappa^j r),\\
\lim_{i\to \infty} q_{x_0,\kappa^i r}=u(x_0),\\
\abs{u(x_0)-q_{x_0,r}} \lesssim \phi(x_0, r)+
\norm{u}_{L^\infty(\Omega(x_0,r))} \int_0^r \frac{\omega_{\mathbf A}(t)}t \,dt+1,
\end{gather*}
respectively.
Also, by averaging the inequality
\[
\abs{q_{x_0, r}}^{\frac12} \le \abs{u(x) -q_{x_0, r}}^{\frac12} + \abs{u(x)}^{\frac12}
\]
over $x \in \Omega(x_0, r)$, taking the square, and using H\"older's inequality, we get
\[
\abs{q_{x_0, r}} \le 2 \phi(x_0, r) + 2 \left(\fint_{\Omega(x_0,r)} \abs{u}^{\frac12}\right)^2 \lesssim \fint_{\Omega(x_0,r)} \abs{u}.
\]
By combining these inequalities, we have \eqref{eq14.39t}.
The rest of proof is the same.

\section{Appendix}				%\label{appendix}
We sketch the proof of the $W^{2,p}$-solvability of \eqref{eq10.36} in a bounded $C^{1,\alpha}$ domain $\Omega\subset \bR^n$ with $\alpha>1-1/\max\{n,p\}$, which enable us to relax the $C^{1,1}$ condition of $\Omega$ in Theorems \ref{thm1} to $C^{1,\alpha}$, where $\alpha>1/2$.
Recall the regularized distance function $\psi$ on $\Omega$ introduced in \cite{Lie85} is such that $\psi(x)$ is comparable to $\text{dist}\,(x,\partial\Omega)$ near the boundary and $\psi\in C^{1,\alpha}(\overline{\Omega})\cap C^\infty(\Omega)$. Without loss of generality, we assume $0\in \partial\Omega$ and the $x_n$-direction is the normal direction at $0$. Taking a small constant $r>0$, we flatten the boundary $\partial\Omega$ near $0$ by making the change of variables
\begin{equation*}
x\in \Omega_r:=\Omega(0,r) \rightarrow y\in \bR^n_{+},\quad y_i(x) = x_i,\,\,i=1,\ldots,n-1,\quad y_n(x)=\psi(x).
\end{equation*}
In the $y$-variables, the equation becomes
\begin{equation*}
 \tilde a_{kl} D_{y_k y_l}u+ \tilde b D_{y_n}  u = f \quad \text{in}\,\,\psi(\Omega_r)\subset\bR^n_{+}
\end{equation*}
with the Dirichlet boundary condition $u=0$ on $\{y_n=0\}$,
where
\[
\tilde a_{kl}=a_{ij}D_{x_i}{y_k} D_{x_j}{y_l},\quad \tilde b=a_{ij}D_{x_i x_j}{\psi}.
\]
Since $\psi\in C^{1,\alpha}(\overline{\Omega})$, $\tilde a_{kl}$ is uniformly continuous in $\psi(\Omega_r)$.
It follows from \cite[Lemma 2.4]{Safonov} that $|b(y)|\le Cy_n^{1-\alpha}$.
In particular, $b\in L^{\max\{n,p\}+\varepsilon}(\psi(\Omega_r))$ for some $\epsilon>0$ provided that $\alpha>1-1/\max\{n,p\}$.
By using the boundary $W^{2,p}$-estimate for elliptic equation,
%with $L^{\max\{n,p\}^+}$ first-order coefficients,
Gerhardt's inequalities (cf. \cite{G79} or \cite[Lemma 1 (ii)]{K09}), and a standard iteration argument, we conclude that for any $W^{2,p}$-solution $u$,
\[
\|u\|_{W^{2,p}(\Omega_{r/2})}\le C\|f\|_{L^p(\Omega_r)}+C\|u\|_{L^p(\Omega_r)}.
\]
It then follows from a partition of unity argument and the corresponding interior estimate that
\[
\|u\|_{W^{2,p}(\Omega)}\le C\|f\|_{L^p(\Omega)}+C\|u\|_{L^p(\Omega)}.
\]
Now by the proof of \cite[Lemma 9.17]{GT98}, using a compactness argument and the Alexandrov maximum principle, we have
\[
\|u\|_{W^{2,p}(\Omega)}\le C\|f\|_{L^p(\Omega)},
\]
which also gives the uniqueness of solutions.
Finally, the existence of solutions follows from an approximation and bootstrap argument. See the proof of \cite[Theorem 9.15]{GT98}.

\section*{Acknowledgement}
The authors would like to thank Zongyuan Li for pointing out the $W^{2,p}$ solvability in Appendix.
We also thank the referee for careful reading of the manuscript and asking to clarify the dependence of various estimates on $\Omega$.

%------------------------------------------------------------------------------%

\end{document}